\documentclass[11pt]{amsart}
\usepackage{amsmath}
\usepackage{amsfonts}
\usepackage{amsthm}
\usepackage{amssymb}
\usepackage{mathrsfs}
\usepackage{latexsym}
\usepackage{verbatim}
\usepackage{stmaryrd}
\usepackage{a4wide}
\usepackage{enumerate}
\usepackage{lipsum}
\usepackage{mathtools} 
\usepackage{enumitem}
\usepackage{enumerate}

\numberwithin{equation}{section}

\renewcommand{\thetheoremName}

\def\Z{\mathbf{Z}}
\def\G{\mathbb{G}}

\def\Q{\mathbf{Q}}
\def\F{\mathbf{F}}
\def\R{\mathbf{R}}

\def\C{\mathbf{C}}

\def\sE{\mathbf{\widetilde{E}}}
\def\sA{\mathbf{\widetilde{A}}}
\def\sB{\mathbf{\widetilde{B}}}

\newtheorem{theorem}{Theorem}[section]
\newtheorem*{theorem*}{Main Theorem}

\newtheorem{df}[theorem]{Definition}
\newtheorem{lemma}[theorem]{Lemma}
\newtheorem{prop}[theorem]{Proposition}

\newtheorem{remark}[theorem]{Remark}
\newtheorem{assumption}[theorem]{Assumption}

\newtheorem*{Lubin}{Lubin's Conjecture}

\begin{document}

\begin{abstract} Let ~$f$ be a continuous ring endomorphism of ~$\mathbf{Z}_p\llbracket      x\rrbracket/\mathbf{Z}_p$ of degree ~$p.$ We prove that if ~$f$ acts on the tangent space at ~$0$ by a uniformizer and commutes with an automorphism of infinite order, then it is necessarily an endomorphism of a formal group over ~$\mathbf{Z}_p.$ The proof relies on finding a stable embedding of ~$\mathbf{Z}_p\llbracket      x\rrbracket$ in Fontaine's crystalline period ring with the property that ~$f$ appears in the monoid of endomorphisms generated by the Galois group of ~$\mathbf{Q}_p$ and crystalline Frobenius. Our result verifies, over ~$\mathbf{Z}_p,$ the height one case of a conjecture by Lubin.       
\end{abstract}

\title{The Crystalline Period of a Height One ~$p$-Adic Dynamical System over ~$\Z_p$}

\subjclass[2010]{11S20, 11S31, 11S82, 14L05, 13F25, 14F30.}
\author{Joel Specter}\thanks{The author was supported in part by National Science Foundation Grant DMS-1404620 and by a National Science Foundation Graduate Research Fellowship under Grant No. DGE-1324585.}
\address{Joel Specter, Department of Mathematics, Northwestern University, 2033 Sheridan Road, Evanston, IL 60208, United States}
\email{jspecter@math.northwestern.edu}

\maketitle

\section{Introduction}  

Given a formal group law ~$F/ \Z_p$ the endomorphism ring of ~$F$ provides an example of a nontrivial family of power series over ~$\Z_p$ which commute under composition. This paper investigates the question of the converse: under what conditions do families (in our case pairs) of commuting power series arise as endomorphisms of an integral formal group?  Let ~$p$ be a prime number, ~$\C_p$ be the completion of an algebraic closure of ~$\Q_p$, and ~$\mathfrak{m}_{\C_p} \subset \C_p$ be the open unit disk. Define ~$S_0(\Z_p)$ to be the set of formal power series over ~$\Z_p$ without constant term.\footnote{In general, for any ring ~$R,$ we define ~$S_0(R)$ to to be the set formal of power series over ~$R$ without constant term.} Substitution defines a composition law on ~$S_0(\Z_p)$ under which it is isomorphic to the monoid of endomorphisms of the formal ~$p$-adic unit disk which fix ~$0.$

\begin{theorem*} Let ~$f$ and ~$u$ be a commuting pair of elements in ~$S_0(\Z_p)$. If
\begin{itemize}
\item ~$f'(0)$ is prime in ~$\Z_p$ and ~$f$ has exactly ~$p$ roots in ~$\mathfrak{m}_{\C_p},$
\item ~$u$ is invertible and has infinite order,
\end{itemize} 
then there exists a unique formal group law ~$F/\Z_p$ such that ~$f,u \in \mathrm{End}_{\Z_p}(F).$ The formal group law ~$F$ is isomorphic to ~$\widehat{\G}_m$ over the ring of integers of the maximal unramified extension of ~$\Q_p.$       
\end{theorem*}

The study of analytic endomorphisms of the ~$p$-adic disk was initiated by Lubin in \cite{L}. There he showed that if ~$f\in S_0(\Z_p)$ is a non-invertible transformation which is nonzero modulo ~$p$ and ~$u\in S_0(\Z_p)$ is an invertible, non-torsion transformation, then assuming ~$f$ and ~$u$ commute:

\begin{enumerate}
\item\label{itm:roots} The number of roots of ~$f$ in ~$\mathfrak{m}_{\C_p}$ is a power of ~$p$ (cf. \cite[Main Theorem 6.3, p. 343]{L}). 
\item\label{itm:reduction} The reduction of ~$f$ modulo ~$p$ is of the form ~$a(x^{p^h})$ where ~$a \in S_0(\F_p)$ is invertible (cf. \cite[Corollary 6.2.1, p. 343]{L}).\footnote{The analogous statements both hold when ~$\Z_p$ is replaced by the ring of integers in a finite extension of ~$\Q_p.$}  
\end{enumerate}

Both (\ref{itm:roots}) and (\ref{itm:reduction}) are well known properties of a non-invertible endomorphism of a formal group law over ~$\Z_p$. In light of this, and knowing: 

\begin{enumerate}[resume]
\item Every non-torsion, invertible power series ~$\mathfrak{u}\in S_0(\Q_p)$ is an endomorphism of a unique one dimensional formal group law ~$F_\mathfrak{u}/\Q_p.$ If ~$\mathfrak{f} \in S_0(\Q_p)$ commutes with ~$\mathfrak{u}$ then ~$\mathfrak{f}\in \mathrm{End}_{\Q_p}(F_\mathfrak{u}),$ 
\end{enumerate}

\noindent one might preliminarily conjecture that the formal group ~$F_u$ associated to ~$u$ is defined over ~$\Z_p.$ This is false and counterexamples have been constructed by Lubin ~\cite[p. 344]{L} and Li ~\cite[p. 86-87]{Li2}. Instead, Lubin hypothesizes  that ``for an invertible series to commute with a non-invertible series, there must be a formal group somehow in the background'' \cite[p. 341]{L}. Interestingly, all known counterexamples are constructed from the initial data of a formal group defined over a finite extension of ~$\Z_p$ and therefore conform to Lubin's philosophy. Despite this, the ``somehow'' in Lubin's statement has yet to be made precise. 

Result (\ref{itm:reduction}) is reminiscent of a hypothesis in the following integrality criterion of Lubin and Tate:    

\begin{enumerate}[resume]
\item \label{itm:Lubin-Tate} Let ~$\mathfrak{f} \in S_0(\Z_p).$ Assume ~$\mathfrak{f}'(0)$ is prime in ~$\Z_p$ and ~$\mathfrak{f}(x) \equiv x^{p^h} \mod p,$ then there is a unique formal group law ~$F_{\mathfrak{f}}/\Z_p$ such that ~$\mathfrak{f} \in \mathrm{End}_{\Z_p}(F)$ (cf. \cite[Lemma 1, p. 381-2]{LT}).
\end{enumerate}

\noindent Perhaps based on this and the counterexamples of \cite{Li2} and \cite{L}, Lubin has offered the following conjecture as to when our preliminary intuition should hold true: 

\begin{Lubin} \cite[p. 131]{S1} Suppose that ~$\mathfrak{f}$ and ~$\mathfrak{u}$ are a non-invertible and a non-torsion invertible series, respectively, defined over the ring of integers ~$\mathcal{O}$ of a finite extension of ~$\Q_p.$ Suppose
further that the roots of ~$\mathfrak{f}$ and all of its iterates are simple, and that
$\mathfrak{f}'(0)$ is a uniformizer in ~$\mathcal{O}$. If ~$\mathfrak{f}\circ \mathfrak{u} = \mathfrak{u} \circ \mathfrak{f}$ then ~$\mathfrak{f}, \mathfrak{u} \in \mathrm{End}_{\mathcal{O}}(F)$ for some formal group law ~$F/\mathcal{O}.$
\end{Lubin}

For each integral extension of ~$\Z_p,$ Lubin's conjecture is naturally divided into cases: one for each possible height  of the commuting pair ~$\mathfrak{f}$ and ~$\mathfrak{u}.$ This note serves as a complete solution to the height one case of Lubin's conjecture over ~$\Z_p.$ Previous results towards Lubin's conjecture (\cite{LMS}, \cite{S}, \cite{SS}) have used the field of norms equivalence to prove some special classes of height one commuting pairs over ~$\Z_p$ arise as endomorphisms of integral formal groups. All other cases of Lubin's conjecture, i.e. those over proper integral extensions of ~$\Z_p$ or over ~$\Z_p$ with height greater than one, are completely open. 

We call a formal group Lubin-Tate (over ~$\Z_p$) if it admits an integral endomorphism ~$\mathfrak{f}$ satisfying the hypotheses of (\ref{itm:Lubin-Tate}). All height one formal groups over ~$\Z_p$ are Lubin-Tate. Therefore, given the height one commuting pair ~$(f,u),$ one strategy to prove that the formal group ~$F_u$ is integral is to identify the Lubin-Tate endomorphism of ~$F_u$ and appeal to result (\ref{itm:Lubin-Tate}). This is the approach we employ in this paper.

Before we present an outline of the proof, we remind the reader of some structures which exist in the presence of a formal group. Let ~$F$ be a height one formal group over ~$\Z_p.$ Then ~$F$ is Lubin-Tate, so there exists a unique endomorphism ~$\mathfrak{f}$ of ~$F$ such that ~$\mathfrak{f}(x)\equiv x^p \mod p.$ Consider the ~$\mathfrak{f}$-adic Tate module  ~$T_{\mathfrak{f}}F.$ The absolute Galois group of ~$\Q_p$ acts on ~$T_{\mathfrak{f}}F$ through a character ~$\chi_F:G_{\Q_p} \rightarrow \Z_p^{\times}.$ 

The character ~$\chi_F$ is crystalline. We sketch a proof of this fact following \cite{KR}. Let ~$\mathcal{O}_{\C_p}$ be the integral closure of ~$\Z_p$ in ~$\C_p$ and ~$\sE^+ := \varprojlim \mathcal{O}_{\C_p}/p,$ where the limit is taken with respect to the ~$p$-power Frobenius map.  Because ~$\mathfrak{f}(x) \equiv x^p \mod p,$ elements of ~$T_{\mathfrak{f}}$ are canonically identified with elements of ~$\sE^+.$ Let ~$v$ be a generator of ~$T_{\mathfrak{f}}.$  There is a unique element ~$\widehat{v} \in \sA^+ := W(\sE^+),$ the Witt vectors of ~$\sE^+,$ which lifts ~$v$ and with the property that the action of ~$G_{\Q_p}$ and the Frobenius endomorphism ~$\varphi$ satisfy:
\begin{enumerate}[label=(\Alph*)]
\item\label{itm:property A} ~$g(\widehat{v}) = [\chi_F(g)]_F(\widehat{v})$ for all ~$g \in G_{\Q_p},$
\item\label{itm:property B} ~$\varphi(\widehat{v}) = \mathfrak{f}(\widehat{v})$ \cite[Lemma~1.2]{KR}.   
\end{enumerate}
The logarithm of ~$F$ evaluated at ~$\widehat{v}$ converges in Fontaine's crystalline period ring ~$\mathbf{B}_{cris}.$ The resulting period ~$\log_F(\widehat{v})$ spans a ~$G_{\Q_p} \times \langle \varphi \rangle$-stable ~$\Q_p$-line on which ~$G_{\Q_p}$ acts through ~$\chi_F.$ 
 
The element ~$\widehat{v} \in \mathbf{B}_{cris}$ is transcendental over ~$\Z_p$ and therefore property \ref{itm:property B} implies that the Lubin-Tate endomorphism ~$\mathfrak{f}$ can be recovered from the action of Frobenius on ~$\widehat{v}.$ Inspired by this observation, we construct in this note an appropriate substitute for ~$\widehat{v}$ which is directly accessible from the commuting pair ~$(f,u).$ The argument is as follows. In section \ref{sect: log and roots}, we recall some preliminary facts concerning the commutative height one ~$p$-adic dynamical system ~$(f,u)$. In particular, we recall the \textit{logarithm}, ~$\log_f,$ of the commuting pair. Next, in section \ref{sect: a character} we attach to ~$(f,u)$ a character ~$\chi_f:G_K\rightarrow u'(0)^{\Z_p}\subseteq \Z_p^{\times}$ where ~$K$ is some particular finite extension of ~$\Q_p.$ The character ~$\chi_f$ is (\textit{a posteriori}) a restriction of the character attached to the Tate module of the formal group for which ~$f$ and ~$u$ are endomorphisms. Then, in section ~\ref{sect: pHT}, we show the character ~$\chi_f$ is crystalline of weight one. This is achieved by constructing an element ~$x_0 \in \sA^+$ such that ~$[\chi_f(g)]_f(x_0) = g(x_0),$ where ~$[\chi_f(g)]_f$  is the unique ~$\Z_p$-iterate of ~$u$ with linear term ~$\chi_f(g).$ We show the logarithm converges in ~$\mathbf{A}_{cris}$ when evaluated at ~$x_0$ and the resulting period ~$t_f := \log_f(x_0)$ generates a ~$G_{K}$-stable ~$\Q_p$-line ~$V_f$ of exact filtration one. Multiplicity one guarantees ~$V_f$ is stable under crystalline Frobenius. Taking ~$\pi_f$ to denote the crystalline Frobenius eigenvalue on ~$V_f,$ we show ~$[\pi_f]_f,$ the multiplication by ~$\pi_f$  endomorphism of ~$F_u,$ converges when evaluated at ~$x_0$ and satisfies  ~$[\pi_f]_f(x_0) = \varphi(x_0).$ The proof is concluded by showing ~$[\pi_f]_f$ is a Lubin-Tate endomorphism of ~$F_u$ and therefore ~$F_u$ is defined over ~$\Z_p$ by ~$(\ref{itm:Lubin-Tate}).$   

\begin{remark} \emph{ There is a recent preprint of Berger \cite{B3} which uses similar methods to study certain iterate extensions. Berger's work is related to the phenomena which occur when the hypotheses of Lubin and Tate's result ~$(\ref{itm:Lubin-Tate})$ are weakened and one assumes merely that ~$\mathfrak{f}$ is a lift of Frobenius i.e. ~$\mathfrak{f}'(0)$ is not assumed to be prime. This paper can be seen as an orthogonal generalization of Lubin and Tate's hypotheses.}  
\end{remark} 

\section{The ~$p$-adic Dynamical System and Its Points} 

\subsection{The Logarithm and Roots}\label{sect: log and roots}

Throughout this paper, we will assume ~$f,u \in S_0(\Z_p)$ are a fixed pair of power series satisfying: 

\begin{assumption}\label{lubin's conditions}
Assume: \begin{itemize}
\item ~$f \circ u = u \circ f,$ 
\item ~$u$ is invertible and has infinite order, and
\item ~$v_p(f'(0)) = 1$ and ~$f$ has exactly ~$p$ roots in ~$\mathfrak{m}_{\C_p}.$
\end{itemize}
\end{assumption}

The goal of this section is to remind the reader of some of the properties of the following two structures associated by Lubin (in \cite{L}) to the commuting pair ~$(f,u).$ The first is the set of ~$f$-preiterates of ~$0$: ~$$\Lambda_f := \{\pi \in \mathfrak{m}_{\C_p}: f^{\circ n}(\pi) = 0 \text{ for some } n \in \Z_+ \}.$$ The second is the \textit{logarithm} of ~$f;$ this is the unique series ~$\log_f \in S_0(\Q_p)$ such that ~$$\log_f'(0) = 1 \text{ and } \log_f(f(x)) = f'(0)\log_f(x).$$  
In the context that ~$f$ and ~$u$ are endomorphisms of a formal group ~$F/\Z_p,$ the set ~$\Lambda_f$ is the set of ~$p^{\infty}$-torsion points of ~$F$ and ~$\log_f$ is the logarithm of ~$F.$

One essential tool for understanding a ~$p$-adic power series and its roots is its Newton polygon. Given a series, ~$$g(x) := \sum_{i=0}^{\infty} a_ix^i  \in \C_p\llbracket      x\rrbracket,$$ the Newton polygon ~$\mathcal{N}(g)$ of ~$g$ is the convex hull in the ~$(v,w)$-Cartesian plane, ~$\R^2,$ of the set of vertical rays extending upwards from the points ~$(i,v_p(a_i)).$ Let ~$pr_1:\R^2 \rightarrow \R$ denote the projection map to the first coordinate. The boundary of ~$\mathcal{N}(g)$ is the image of an almost everywhere defined piecewise linear function ~$B_g: pr_1(\mathcal{N}(g)) \rightarrow \R^2.$ The derivative of ~$B_g$ is almost everywhere defined and increasing.  The points of the boundary of ~$\mathcal{N}(g)$ where the slope jumps are called the vertices of ~$\mathcal{N}(g),$ whereas the maximal connected components of the boundary of ~$\mathcal{N}(g)$ where the slope is constant are called the segments. The width of a segment is the length of its image under ~$\pi_1.$ The following data can be read off from the Newton polygon of ~$g:$      

\begin{theorem}\label{convergence} The radius of the maximal open disk in ~$\C_p$ centered at ~$0$ on which ~$g$ converges is equal to ~$\displaystyle\lim_{v \rightarrow \infty} p^{\frac{dB_g}{dv}}.$ 
\end{theorem}

\begin{theorem} The series ~$g$ has a root in its maximal open disk of convergence of valuation ~$\lambda$ if and only if the Newton polygon of ~$g$ has a segment of slope ~$-\lambda$ of finite, positive width. The width of this slope is equal to the number of roots of ~$g$ (counting multiplicity) of valuation ~$\lambda.$    
\end{theorem}

\begin{theorem}[Weierstrass Preparation Theorem]\label{WP} If the series ~$g$ is defined over a complete extension ~$L/\Q_p$ contained in ~$\C_p$ and the Newton polygon of ~$g$ has a segment of slope ~$-\lambda$ of finite, positive width then ~$g = g_1g_2$ where ~$g_1$ is a monic polynomial over ~$L$ and ~$g_2 \in L\llbracket      x\rrbracket,$ such that the roots of ~$g_1$ are exactly the roots of ~$g$ (with equal multiplicity) of valuation ~$\lambda.$ We call ~$g_1$ the Weierstrass polynomial corresponding to the slope ~$-\lambda.$      
\end{theorem}

For proofs of these statements we refer the reader to \cite{K}. 

In this paper, we will be concerned with roots of power series in ~$\mathfrak{m}_{\C_p}$ and therefore focus on segments of Newton polygons of negative slope. 

\begin{prop}\label{Newton shape} The endpoints of the segments of the Newton polygon of ~$f^{\circ n}$ which have negative slope are ~$(p^k,n-k)$ for each integer ~$k$ with ~$0 \leq k \leq n.$  
\end{prop}

\begin{proof} We prove the claim by induction on ~$n.$ We begin by proving the claim in the case that ~$n=1.$ Because ~$f(0) = 0$ and ~$f'(0) \neq 0,$ the Newton polygon of ~$f$ has a vertex at ~$(1,v_p(f'(0))) = (1,1).$ Since ~$f$ was assumed to have ~$p$ roots in ~$\mathfrak{m}_{\C_p},$ the initial finite slope of ~$\mathcal{N}(f)$ must be negative. Furthermore, because ~$f$ is defined over ~$\Z_p,$ the Newton polygon of ~$f$ is contained in the quadrant ~$\R_{\geq 0}^2$ and all vertices have integer coordinates. Therefore, the end of the initial segment of ~$\mathcal{N}(f)$ must have a vertex on the ~$v$-axis. By the Weierstrass Preparation Theorem, the nonzero roots of ~$f$ corresponding to this segment satisfy an Eisenstein polynomial and as such they are simple. Since ~$f$ has exactly ~$p$ roots in ~$\mathfrak{m}_{\C_p},$ the Newton polygon of ~$f$ has a vertex at ~$(p,0)$ and all nonzero roots in ~$\mathfrak{m}_{\C_p}$ are accounted for by the initial finite segment of ~$\mathcal{N}(f)$. 

Let ~$n \geq 1.$ The base case of our inductive argument already tells us a lot about ~$f^{\circ n+1}.$ For example because ~$\mathcal{N}(f)$ has ~$(0,p)$ has a vertex on the ~$v$-axis, ~$$f(x) \equiv a_px^p \mod (p, x^{p+1}),$$ where ~$a_p \in \F_p^{\times}.$  It follows ~$$f^{\circ {n+1}}(x) \equiv a_p^{n+1}x^{p^{n+1}} \mod (p, x^{p^{n+1}+1}).$$ Therefore, ~$\mathcal{N}(f^{\circ n+1})$ has a vertex at ~$(0,p^{n+1}).$ Similarly, because ~$\mathcal{N}(f)$ has a vertex at ~$(1,1)$ and ~$f$ has a fixed point at ~$0,$ the Newton polygon ~$\mathcal{N}(f^{\circ n+1})$ has a vertex at ~$(1,n+1).$ The ~$v$ coordinate of a vertex abutting any finite, negatively sloped segment of ~$\mathcal{N}(f^{\circ n+1})$ must occur between ~$1$ and ~$p^{n+1}.$      

Now assume we have shown the claim for ~$n.$ The roots of ~$f^{\circ n+1}$ in ~$\mathfrak{m}_{\C_p}$ are either roots of ~$f^{\circ n}$ or roots of ~$f^{\circ n}(x) - \pi$ where ~$\pi$ is one of the ~$p-1$ nonzero roots of ~$f.$ For any nonzero root ~$\pi$ of ~$f,$ the Newton polygon of ~$f^{\circ n}(x) - \pi$ consists a single segment of slope ~$-1/(p^{n+1} - p^n)$ and width ~$p^n.$ Thus ~$\mathcal{N}(f^{\circ {n+1}})$ contains a segment of slope ~$-1/(p^{n+1} - p^n).$ This slope is shallower than the slope of any segment of ~$\mathcal{N}(f^{\circ n})$ and therefore the segment of ~$\mathcal{N}(f^{\circ {n+1}})$ of slope ~$-1/(p^{n+1} - p^n)$ must be the final negatively sloped segment. For each nonzero root ~$\pi$ of ~$f,$ the Weierstrass polynomial corresponding to this single negatively sloped segment of ~$\mathcal{N}(f^{\circ n} (x)- \pi)$ is Eisenstein over ~$\Z_p[\pi]$ and therefore its roots are distinct. Because the nonzero roots of ~$f$ are distinct,  there are exactly ~$p^n(p-1)$ roots of ~$f^{\circ n+1}$ of valuation ~$-1/(p^{n+1} - p^n)$ and therefore ~$\mathcal{N}(f^{\circ {n+1}})$ has a vertex at ~$(1,p^n).$  

The remaining negative segments of ~$\mathcal{N}(f^{\circ {n+1}})$ must correspond to roots of ~$f^{\circ n+1}$ which are roots of ~$f^{\circ n}.$ By the inductive hypothesis, one can deduce the segments of ~$\mathcal{N}(f^{\circ n})$ of finite, negative slope have width ~$p^k - p^{k-1}$ and slope ~$-1/(p^k - p^{k-1})$ where ~$k$ runs over the integers ~$1\leq k\leq n.$ Hence, all Weierstrass polynomials of ~$f^{\circ n}$ are Eisenstein over ~$\Z_p$ and therefore all roots of ~$f^{\circ n}$ in ~$\mathfrak{m}_{\C_p}$ are simple. It follows, since every root of ~$f^{\circ n}$ in ~$\mathfrak{m}_{\C_p}$ is a root of ~$f^{\circ n+1},$ the Newton polygon ~$\mathcal{N}(f^{\circ n+1})$ contains for each segment of ~$\mathcal{N}(f^{\circ n})$ a segment of the same slope and at minimum the same width. These segments span between the vertices ~$(1,n+1)$ and ~$(1,p^n).$ The only way for this to occur is if the boundary of the Newton polygon of ~$\mathcal{N}(f^{\circ n+1})$ in this range is equal to the boundary Newton polygon of ~$\mathcal{N}(f^{\circ n})$ translated up by one in the positive ~$w$ direction. The claim follows by induction.\end{proof}

Proposition \ref{Newton shape} gives us a good understanding of the position of ~$\Lambda_f$ in the ~$p$-adic unit disc. Unpacking the data contained in the Newton polygon of ~$f^{\circ n}$ for each ~$n,$ we see that ~$\Lambda_f$ is the disjoint union of the sets 

$$C_n := \{\pi \in \Lambda_f: v_p(\pi) = \frac{1}{p^n -p^{n-1}}\}$$ 

\noindent for each ~$n \geq 1$ and ~$C_0 := \{0\}.$ For ~$n \geq 1,$ the set ~$C_n$ has cardinality ~${p^n -p^{n-1}}$ and its elements are the roots of a single Eisenstein series over ~$\Z_p.$ Each of these are expected properties of the ~$p^{\infty}$-torsion of a height one formal group over ~$\Z_p.$ In that case, the set ~$C_n$ is the set of elements of exact order ~$p^n.$ The same holds in the absence of an apparent formal group. Comparing the Newton polygons of ~$f^{\circ n}$ and ~$f^{\circ n +1},$ we note ~$C_n$ for ~$n \geq 1$ can alternatively be described as: 

$$C_n = \{x \in \mathfrak{m}_{\C_p}: f^{\circ n}(x) = 0 \text{ but } f^{\circ n-1}(x) \neq 0 \}.$$

\noindent Because ~$f$ and ~$u$ commute, the closed subgroup of ~$S_0(\Z_p)$ topologically generated by ~$u$ acts on ~$\Lambda_f$ and preserves the subsets ~$C_n.$ Simultaneously, because ~$f$ and ~$u$ are defined over ~$\Z_p,$ the Galois group ~$G_{\Q_p},$ the absolute Galois group of ~$\Q_p,$ acts on ~$\Lambda_f,$ preserves the subsets ~$C_n,$ and commutes with the action of ~$u^{\circ \Z_p}.$ In the next section, we will explore these commuting actions in more detail.        

Next we recall the \textit{logarithm} of a the commuting pair ~$(f,u).$ Given any series ~$g \in S_0(\Z_p)$ such that ~$g'(0)$ is nonzero nor equal to a root of unity there exists a unique series ~$log_g \in S_0(\Q_p)$ such that, 
\begin{enumerate}
\item ~$\log_g'(0) = 1,$ and  
\item ~$\log_g(g(x)) = g'(0)\log_g(x)$ \noindent \cite[Proposition~1.2]{L}.
\end{enumerate}
\noindent Lubin shows that two series ~$g_1,g_2 \in S_0(\Z_p)$ such that ~$g_1'(0)$ and  ~$g_2'(0)$ satisfy the above condition have equal logarithms if and only if they commute under composition   \cite[Proposition~1.3]{L}.

Consider the sequence of series in ~$\Q_p\llbracket      x\rrbracket$ whose ~$n$-th term is:
$$\frac{f^{\circ n}(x)}{f'(0)^n}.$$   Lubin shows that this sequence converges coefficient-wise to a series ~$\log_f$ satisfying: 
\begin{enumerate}
\item ~$\log_f'(0) = 1$  
\item ~$\log_f(f(x)) = f'(0)\log_f(x)$ \noindent \cite[Proposition~2.2]{L}.\footnote{One can place  a topology  on ~$\Z_p\llbracket      X\rrbracket$ induced by a set of valuations arising from Newton copolygons. This topology is strictly finer than the topology of coefficient-wise convergence.  Lubin proves that the limit defining ~$\log_f$ converges in the closure of ~$\Z_p\llbracket      X\rrbracket$ under this topology.} 
\end{enumerate} 
The limit is the logarithm of ~$(f,u).$ It is the unique series over ~$\Q_p$ with these properties \cite[Proposition~1.2]{L}.

From proposition \ref{Newton shape}, one obtains that the vertices of Newton polygon of ~$\log_f$ are ~$(p^k,-k)$ where ~$k$ ranges over the nonnegative integers. Therefore by Theorem \ref{convergence}, ~$\log_f$ converges on ~$\mathfrak{m}_{\C_p}.$ Furthermore, the Newton polygon of  ~$\log_f$ displays that ~$\log_f$ has simple roots (as each of its Weierstrass  polynomials is Eisenstein over ~$\Z_p$) and the root set of ~$\log_f$ and ~$\Lambda_f$ have the same number of points on any circle in ~$\mathfrak{m}_{\C_p}.$ One guesses that the root set of ~$\log_f$ (counting multiplicity) is ~$\Lambda_f.$ This is true and proved by Lubin \cite[Proposition~2.2]{L}.

The logarithm of ~$f$ is invertible in ~$S_0(\Q_p).$ Conjugation by ~$\log_f$ defines a continuous, injective homomorphism from the centralizer of ~$f$ in  ~$S_0(\Z_p)$ to ~$\mathrm{End}_{\Z_p}(\widehat{\G}_a).$ This map sends the power series ~$e$ to ~$e'(0)x.$      

Consider the subgroup ~$u^{\circ \Z_p} \subseteq S_0(\Z_p).$ By the previous paragraph, conjugation by ~$\log_f$ identifies this group with the closed subgroup ~$ u'(0)^{ \Z_p}$ of ~$\Z_p^{\times}.$ By assumption \ref{lubin's conditions}, this group is infinite. Replacing ~$u$ with an appropriate finite iterate, we may assume without loss of generality that:

\begin{assumption}\label{u assumption} The group  ~$u^{\circ \Z_p}$ is topologically isomorphic to ~$\Z_p$ and   ~$$v_p((1- u'(0))^m) = v_p(1- u'(0)) + v_p(m)$$ for all ~$m \in \Z_p.$
\end{assumption}

%The invariants ~$\Lambda_f$ and ~$\log_f,$ were defined by considering ~$f.$ However, they are better thought of as invariants of the dynamical system ~$(f,u)$ and can be independently defined from the knowledge of ~$u$ alone.  Let 

%$$\Lambda_u :=  \{ \pi \in \mathfrak{m}_{\C_p} : u^{\circ n}(\pi) = \pi \text{ for some } n \in \Z_+ \}$$
%and 
%$$\log_u(x) := \lim_{n\rightarrow \infty} \frac{u^{\circ (p-1)p^n}(x) - x}{u'(0)^{(p-1)p^n}}.$$      

%Lubin shows that all the roots of ~$u$ are simple, the set  ~$\Lambda_u = \Lambda_f,$ and ~$\log_u$ converges in ~$\Q_p\llbracket      x\rrbracket$ and ... In the next section, we will action of ~$u^{\Z_p}$ on ~$\Lambda_f$ to witt a strong knowledge of the fixed points of  ~$u^{\circ (p-1)p^n}$ will be essential.  

%\begin{prop} Let ~$e \in u^{\circ \Z_p}$ be a nonidentity iterate. If ~$v_p(e'(0) - 1) \geq 1,$ then the fixed points of ~$e$ in ~$\mathfrak{m}_{\C_p}$ is   
%\end{prop}

%For any invertible, non-torsion transformation ~$e \in S_0(\Z_p),$  we denote the number fixed points of  ~$e$ in ~$\mathfrak{m}_{\C_p}$ by ~$i(e).$ 

\subsection{A Character Arising from a Commuting Pair}\label{sect: a character}

Let ~$F$ be a height one formal group over ~$\Z_p.$ A fundamental invariant attached to ~$F$ is its Tate module, ~$T_pF.$ Given that ~$F$ is height one, ~$T_pF$ is a rank one ~$\Z_p$-module on which ~$G_{\Q_p}$ acts through a character ~$\chi_F:G_{\Q_p} \rightarrow \Z_p^{\times}.$ This map is surjective. 

Simultaneously, the automorphism group, ~$\mathrm{Aut}_{\Z_p}(F),$ acts on ~$T_pF$ and commutes with the action of ~$G_{\Q_p}.$ This action is through the character which sends an automorphism to multiplication by its derivative at ~$0.$ Because the automorphism group ~$\mathrm{Aut}_{\Z_p}(F)$ is isomorphic to ~$\Z_p^{\times}$ via this character, given any ~$g\in G_{\Q_p}$ there is a unique ~$s_g\in \mathrm{Aut}_{\Z_p}(F)$ such that  for all nonzero ~$v \in T_pF$ the equality ~$g.v = s_g.v$ holds. The value ~$\chi_F(g)$ can be recovered from ~$s_g$ as ~$\chi_F(g) = s'_g(0).$ 

This fact should be considered remarkable for it allows one to recover the character ~$\chi_F$ without any explicit knowledge of the additive structure of ~$T_pF.$ Rather ~$\chi_F$ can be completely deduced from the structure of the ~$G_{\Q_p}$-orbit of any nonzero element ~$v$ as a ~$G_{\Q_p} \times \mathrm{Aut}_{\Z_p}(G)$-set.    

In this section, guided by this observation, we will attach to our commuting pair ~$(f,u)$ a character ~$\chi_f$ from the Galois group of a certain finite extension ~$K$ of ~$\Q_p$ to ~$\Z_p^{\times}.$ The character ~$\chi_f$ arises from the Tate module of the latent formal group for which ~$f$ and ~$u$ are endomorphisms.    

We begin by defining sequences of elements of ~$\mathfrak{m}_{C_p}$ which will substitute for the r\^ole of elements of the Tate module of ~$F.$

\begin{df} \cite[p. 329]{L} An ~$f$-consistent sequence is a sequence of elements ~$(s_1,s_2, s_3, \ldots\ )$ of ~$\mathfrak{m}_{\C_p}$ such that ~$f(s_1) = 0$ and for ~$i>1, f(s_i) = s_{i-1}.$   
\end{df}

Denote the set of all ~$f$-consistent sequences whose first entry is nonzero by ~$\mathcal{T}_0.$

\begin{prop}\label{galois tree} The Galois group ~$G_{\Q_p}$ acts transitively on ~$\mathcal{T}_0.$
\end{prop} 

\begin{proof} The ~$n$-th coordinate of any ~$f$-consistent sequence in ~$\mathcal{T}_0$ is a root of ~$f^{\circ n}$ which is not a root of ~$f^{\circ n-1}.$ It is enough to show that ~$G_{\Q_p}$ acts transitively on these elements. Comparing the Newton polygons of ~$f^{\circ n}$ and ~$f^{\circ n-1},$ we deduce that the set of such elements satisfy a degree ~$p^n - p^{n-1}$ Eisenstein polynomial over ~$\Q_p.$ This polynomial is irreducible and hence ~$G_{\Q_p}$ acts transitively on its roots.     
\end{proof}

Choose an ~$f$-consistent sequence ~$\Pi = (\pi_1,\pi_2, \pi_3, \ldots\ )$ such that ~$\pi_1 \neq 0.$ The sequence ~$\Pi$ will be fixed throughout the remainder of this paper. Because ~$u \in S_0(\Z_p)$ is nontrivial, it can have only finitely many fixed points in ~$\mathfrak{m}_{\C_p}.$ Let ~$k$ be the largest integer such that ~$u(\pi_k) = \pi_k.$ As ~$u^{\circ \Z_p} \cong \Z_p$ is a ~$p$-group and there are ~$p-1$ nonzero roots of ~$f,$ the integer ~$k$ is at least ~$1.$ Let ~$\mathcal{T}_{\pi_k}$ denote the set of all ~$f$-consistent sequences ~$(s_1,s_2, s_3, \ldots\ )$ such that ~$s_i = \pi_i$ for ~$i\leq k.$  

\begin{prop}  ~$\mathcal{T}_{\pi_k}$ is a torsor for the closed subgroup of ~$S_0(\Z_p)$ generated by ~$u.$  
\end{prop}

\begin{proof} The subgroup of ~$S_0(\Z_p)$ topologically generated by ~$u$ is isomorphic to ~$\Z_p.$ Therefore, as ~$\mathcal{T}_{\pi_k}$ is infinite, if ~$u^{\circ \Z_p}$ acts transitively on ~$\mathcal{T}_{\pi_k}$ it must also act freely. To prove this action is transitive, we invoke the counting powers of the orbit stabilizer theorem.  

We begin by calculating the fixed points of iterates of ~$u.$ Consider the set ~$$\Lambda_u :=  \{ \pi \in \mathfrak{m}_{\C_p} : u^{\circ n}(\pi) = \pi \text{ for some } n \in \Z_+ \}$$ of ~$u$-periodic points in ~$\mathfrak{m}_{\C_p}.$ Lubin proves ~$\Lambda_u = \Lambda_f$ \cite[Proposition~4.2.1]{L}. Let ~$e$ be a finite iterate of ~$u.$ Consider the series ~$e(x) - x.$ The roots of this series are the fixed points of  ~$e,$ and hence each root is an element of ~$\Lambda_u.$ Since ~$e$ is defined over ~$\Z_p,$ the set of roots of ~$e$ in ~$\mathfrak{m}_{\C_p}$ is a finite union of  ~$G_{\Q_p}$-orbits in ~$\Lambda_u = \Lambda_f.$ The  ~$G_{\Q_p}$-orbits in ~$\Lambda_f$ are the collection of sets ~$$C_n = \{x \in \mathfrak{m}_{\C_p}: f^{\circ n}(x) = 0 \text{ but } f^{\circ n-1}(x) \neq 0 \},$$ where ~$n$ ranges over the positive integers and ~$C_0 = \{0\}.$ Evaluation under ~$f$ is an ~$e$-equivariant surjection from ~$C_n$ to ~$C_{n-1}.$ Hence the fixed points of ~$e$ are contained in  ~$\bigcup_{i\leq n} C_i$ for some ~$n.$  By \cite[Proposition~4.5.2]{L} and \cite[Proposition~4.3.1]{L}, respectively, the roots of ~$e(x) - x$ are simple and ~$e(x) - x  \not\equiv 0 \mod p.$ Thus, the negative slopes of the Newton polygon of ~$e(x) - x$ must span between  ~$(1,v_p(1-e'(0)))$ and the ~$v$-axis. Since when ~$n\geq 1$ the set ~$C_n$ is the full set of roots of an Eisenstein polynomial ~$\Z_p$, each of the sets ~$C_n$ account for a segment in the Newton polygon of ~$e(x)-x$ with a decline of one unit.  We conclude that the roots of ~$e(x)-x$ and hence the fixed points of ~$e$ are simple and equal to ~$\bigcup_{i\leq v_p(1-e'(0))} C_i.$

From this calculation, we deduce that ~$v_p(1- u'(0)) = k$ and the point-wise stabilizer of ~$\pi \in C_n$ is the set ~$$\{ e\in u^{\circ \Z_p}: v_p(1-e'(0)) \geq n\}.$$  By assumption \ref{u assumption}, it follows  the stabilizer of  any ~$\pi \in C_n$  for ~$n \geq k$ is the group ~$u^{\circ p^{n-k}\Z_p}.$

We now show that ~$u^{\circ \Z_p}$ acts transitively on ~$\mathcal{T}_{\pi_k}.$ As in the proof of proposition \ref{galois tree}, it is enough to show that ~$u^{\circ \Z_p}$ acts transitively on the set of possible values for the ~$n$-th coordinate of a sequence in ~$\mathcal{T}_{\pi_k}.$ Let ~$n\geq k$ and consider the set ~$$W_n(\pi_k) = \{\pi \in \mathfrak{m}_{\C_p}: f^{\circ n - k}(\pi) = \pi_k \}$$ The set ~$W_n(\pi_k)$ are the possible values of an n-th coordinate of a sequence in  ~$\mathcal{T}_{\pi_k}.$ Because ~$\pi_k \neq 0,$ the Newton polygon of ~$f^{\circ n - k} - \pi_k$ has exactly one segment of positive slope. The length of this segment is ~$p^{n-k}$ and its slope is ~$\frac{1}{p^{n-1}(p-1)}.$ Hence, ~$W_n(\pi_k)$ has order ~$p^{n-k}$ and is contained in ~$C_{n}.$ It follows that the ~$u^{\circ \Z_p}$-orbit of a point ~$\pi \in W_n(\pi_k)$ has order ~$$|u^{\circ \Z_p}/u^{\circ p^{n-k} \Z_p}| = p^{n-k} = |W_n(\pi_k)|.$$ We conclude the action of ~$u^{\circ \Z_p}$ on ~$W_n(\pi_k)$ is transitive.\end{proof}

Let ~$K := \Q_p(\pi_k).$ The transitive ~$G_{\Q_p}$ action on ~$\mathcal{T}_0$ restricts to a transitive ~$G_K$-action on ~$\mathcal{T}_{\pi_k}.$ The kernel of this action is ~$K_{\infty} := \Q_p(\pi_i|i>0).$ Let ~$\sigma \in G_K.$ Because  ~$\mathcal{T}_{\pi_k}$ is a torsor for ~$u^{\circ \Z_p}$ there is a unique element ~$u_\sigma \in u^{\circ \Z_p}$ such that ~$u_\sigma\Pi = \sigma\Pi.$ We set ~$\chi_f(\sigma) := u_\sigma'(0).$ The power series ~$u_\sigma$ can be recovered from ~$\chi_f(\sigma)$ as ~$u_\sigma = [\chi_f(\sigma)]_f.$     

\begin{prop} The map ~$[\chi_f]_f: G_K \rightarrow u^{\circ \Z_p}$ is a surjective group homomorphism satisfying ~$[\chi_f(\sigma)]_f(\pi_i) = \sigma(\pi_i)$ for all ~$i>0.$ The kernel of ~$[\chi_f]_f$ is ~$K_{\infty}.$     
\end{prop}

\begin{proof} Outside of the fact that ~$[\chi_f]_f$ is a homomorphism, this proposition follows directly from the discussion above. We show ~$[\chi_f]_f$ is a homomorphism. Let ~$\sigma_1,\sigma_2 \in G_K.$ Then
\begin{align*} [\chi_f(\sigma_2)]_f[\chi_f(\sigma_1)]_f(\Pi)&=[\chi_f(\sigma_1)]_f[\chi_f(\sigma_2)]_f(\Pi) \\&= [\chi_f(\sigma_1)]_f\sigma_2(\Pi) \\&= \sigma_2([\chi_f(\sigma_1)]_f\Pi) \\&= \sigma_2\sigma_1(\Pi).\end{align*}
By uniqueness, it follows ~$[\chi_f(\sigma_2)]_f[\chi_f(\sigma_1)]_f = [\chi_f(\sigma_2\sigma_1)]_f$ and ~$[\chi_f]_f$ is a homomorphism.
\end{proof}

The character ~$\chi_f$ comes via a 'geometric' construction and therefore should have good behavior from the viewpoint of ~$p$-adic Hodge theory. In the next section we show this is the case. Specifically, we show ~$\chi_f$ is crystalline of weight ~$1.$  

\section{The Hodge Theory of ~$\chi_f$}\label{sect: pHT}

\subsection{Fontaine's Period Rings} To show that ~$\chi_f$ is crystalline, we must show that ~$\chi_f$ occurs (by definition) as a ~$G_K$-sub-representation in Fontaine's period ring ~$\mathbf{B}_{cris}.$ In this section, we will recall the construction of this ring as well as several other of Fontaine's rings of periods (for which the original constructions can be found in \cite{F}). We will follow the naming conventions of Berger \cite{B2}.  

The construction of the period rings begin in characteristic ~$p.$ Let ~$\mathbf{\sE}^+ := \varprojlim_{x \mapsto x^p} \mathcal{O}_{\C_p}/p.$ The ring ~$\mathbf{\sE}^+$ inherits an action of the Galois group ~$G_{\Q_p}$ from the action of ~$G_{\Q_p}$ on ~$\mathcal{O}_{\C_p}.$ Additionally, by virtue of being a ring of characteristic ~$p,$ the ring ~$\mathbf{\sE}^+$  comes equipped with a Frobenius endomorphism ~$Frob_{\mathbf{\sE}^+}.$ The map ~$Frob_{\mathbf{\sE}^+}$ commutes with the action of ~$G_{\Q_p}.$ 

The map ~$Frob_{\mathbf{\sE}^+}$ is an isomorphism and therefore ~$\mathbf{\sE}^+$ is perfect. Set ~$\mathbf{\widetilde{A}}^+ = W(\mathbf{\sE}^+),$ the Witt vectors of ~$\mathbf{\sE}^+.$ The formality of the Witt vectors implies that the commuting actions of ~$Frob_{\mathbf{\sE}^+}$ and ~$G_{\Q_p}$ on ~$\mathbf{\sE}^+$ lift, respectively, to commuting actions of ~$Frob_{\mathbf{\sE}^+}$ and ~$G_{\Q_p}$ on ~$\mathbf{\widetilde{A}}^+$ as ring endomorphisms. The lift of ~$Frob_{\mathbf{\sE}^+}$ is denoted by ~$\varphi.$  

Let ~$\sB^+ := \sA^+[\frac{1}{p}].$ Fontaine constructs a surjective ~$G_{\Q_p}$-equivariant ring homomorphism ~$\theta:\sB^+ \rightarrow \C_p$ which can be defined as follows. The map ~$\theta$ is the unique homomorphism which is continuous with respect to the ~$p$-adic topologies on ~$\sB^+$ and ~$\C_p,$ and which maps the Teichm{\"u}ller representative ~$\{x \}$ to ~$\lim_{n \rightarrow \infty} x_n$ where ~$x_n$ is any lift of ~$x^{1/p^n}$ to ~$\mathcal{O}_{\C_p}.$ 

The period ring ~$\mathbf{B}_{dR}^{+}$ is defined as the ~$\ker(\theta)$-adic completion of ~$\sB^+.$ The ring ~$\mathbf{B}_{dR}^{+}$ is topologized so that inherits the topology generated by the intersection of the ~$\ker(\theta)$-adic and ~$p$-adic topologies of the dense subring ~$\sB^+.$ The map ~$\theta$ extends to a continuous ~$G_{\Q_p}$-equivariant surjection ~$\theta:\mathbf{B}_{dR}^{+} \rightarrow \C_p$ and ~$\ker(\theta)$ is principal. The ring ~$\mathbf{B}_{dR}^{+}$ is abstractly isomorphic to ~$\C_p\llbracket      t\rrbracket.$ The map ~$\theta$ induces a ~$\mathbf{N}$-graded filtration on ~$\mathbf{B}_{dR}^{+}$ where ~$\mathrm{Fil}^i\mathbf{B}_{dR}^{+}:= \ker(\theta)^i.$ The de Rham ring of periods, ~$\mathbf{B}_{dR},$ is defined as the fraction field of ~$\mathbf{B}_{dR}^+.$ The filtration on ~$\mathbf{B}_{dR}^{+}$ extends to a ~$\mathbf{Z}$-graded filtration on ~$\mathbf{B}_{dR}.$

Let ~$V/\Q_p$ be a finite dimensional representation of the absolute Galois group of a ~$p$-adic field ~$E.$ We say ~$V$ is de Rham if the ~$E$-dimension of ~$\mathbf{D}_{dR}^*(V) := \mathrm{Hom}_{\Q_p\llbracket      G_E\rrbracket}(V,\mathbf{B}_{dR})$ is equal to the dimension of ~$V.$ The filtration on ~$\mathbf{B}_{dR}$ induces a filtration ~$\mathbf{D}_{dR}^*(V).$ The nonzero graded pieces of this filtration are called the weights of ~$V.$ The weight of the cyclotomic character is ~$1.$

The ring ~$\mathbf{B}_{dR}$ does not admit a natural extension of the Frobenius endomorphism, ~$\varphi,$ of ~$\sA^+.$ To rectify this, Fontaine defines a subfield ~$\mathbf{B}_{cris}$ of ~$\mathbf{B}_{dR}$ on which the endomorphism ~$\varphi$ extends. The field ~$\mathbf{B}_{cris}$ is defined as follows: Let ~$\mathbf{A}^\circ_{cris}$ be the PD-envelope of ~$\sA^+$ with respect to the ideal ~$\ker(\theta) \cap \sA^+.$ The ring ~$\mathbf{A}^\circ_{cris}$ is a subring of ~$\mathbf{B}_{dR}^+.$ We define ~$\mathbf{A}_{cris}$ to be the ~$p$-adic completion of ~$\mathbf{A}^\circ_{cris}.$ One can show that the inclusion of ~$\mathbf{A}^\circ_{cris}$ into ~$\mathbf{B}_{dR}^+$ extends naturally to an inclusion of ~$\mathbf{A}_{cris}$ into ~$\mathbf{B}_{dR}^+.$ Under this inclusion ~$\mathbf{A}_{cris}$ is identified with the set of elements of the form ~$\sum_{i=0}^{\infty} a_n \frac{t^n}{n!}$ where ~$a_n \in \sA^+$ such that ~$\lim_{n\rightarrow \infty} a_n = 0$ in the ~$p$-adic topology and ~$t\in \ker(\theta) \cap \sA^+.$ The ring ~$\mathbf{B}_{cris}^+$ is defined as ~$\mathbf{B}_{cris}^+:= \mathbf{A}_{cris}[\frac{1}{p}]$ and  ~$\mathbf{B}_{cris}$ is the defined as the fraction field of ~$\mathbf{A}_{cris}.$ Let ~$E_0$ be the maximal unramified extension of ~$\Q_p$ contained in ~$E.$ A finite dimensional representation ~$V/\Q_p$ of the absolute Galois group of a ~$p$-adic field ~$E$ is called crystalline if the ~$E_0$-dimension of ~$\mathbf{D}_{cris}^*(V) := \mathrm{Hom}_{\Q_p\llbracket      G_E\rrbracket}(V,\mathbf{B}_{cris})$ is equal to the dimension of ~$V.$ The endomorphism ~$\varphi$ extends uniquely to an endomorphism of the rings ~$\mathbf{A}^\circ_{cris}, \mathbf{A}_{cris}, \mathbf{B}_{cris}^+.$  and ~$\mathbf{B}_{cris}.$

\subsection{The Universal ~$f$-Consistent Sequence}\label{sect:a ring}

Let ~$\Z_p\llbracket      x_1\rrbracket$ be the ring of formal power series over ~$\Z_p$ in the indeterminate ~$x_1.$ In this section we will define a ring,   ~$A_{\infty},$ containing ~$\Z_p\llbracket      x_1\rrbracket,$ which parameterizes ~$f$-consistent sequences in certain topological rings. The initial term of such a sequence will be parameterized by ~$x_1.$ We will show that one can define a continuous injection ~$A_{\infty} \hookrightarrow \sA^+$ under which the image of ~$A_{\infty}$ is ~$G_{K}$ stable and satisfies ~$\sigma(x_1) = [\chi_f(\sigma)]_f(x_1).$ 

We begin by defining ~$A_{\infty}$. For each positive integer ~$i,$ set ~$A_i := \Z_p\llbracket      x_i\rrbracket,$ the ring of formal power series over ~$\Z_p$ in the indeterminate ~$x_i.$ Denote, for each ~$i,$ the homomorphism that maps ~$x_i \mapsto f(x_{i+1})$ by ~$[f'(0)]_f^*: A_i \rightarrow A_{i+1}.$ We define ~$A_{\infty}^{\circ} := \displaystyle\varinjlim A_i$ to be the colimit of the rings ~$A_i$ with respect to the transition maps ~$[f'(0)]_f^*.$ Finally, the ring ~$A_{\infty}$ is defined to be the ~$p$-adic completion of ~$A_{\infty}^{\circ}.$ 

We view ~$A_{\infty}$ as a topological ring under the adic topology induced by the ideal ~$(p,x_1).$ The reader should be aware that while ~$A_{\infty}$ is complete with respect to the finer ~$p$-adic topology, it is not complete with respect to the ~$(p,x_1)$-adic topology. Under the topology on ~$A_{\infty}$ the elements ~$x_i$ are topologically nilpotent and together topologically generate ~$A_{\infty}$ as a ~$\Z_p$-algebra.

The ring ~$A_{\infty}$ is a natural object to consider from the viewpoint of nonarchemedian dynamics for it is universal for ~$f$-consistent sequences in the following sense: given any ~$p$-adically complete, ind-complete adic ~$\Z_p$-algebra ~$S$ and any ~$f$-consistent sequence ~$s := (s_1,s_2,s_3, \dots)$ of topologically nilpotent elements in ~$S$ there exists a unique homomorphism ~$\phi_s: A_{\infty} \rightarrow S$ such that ~$\phi_s(x_i) = s_i.$ Conversely, any homomorphism from ~$A_{\infty}$ to ~$S$ arises as ~$\phi_s$ for some sequence ~$s.$ In other words, ~$A_{\infty}$ represents the functor from ~$p$-adically complete, ind-complete adic ~$\mathbb{Z}_p$-algebras to sets which sends an algebra ~$S$ to the set of ~$f$-consistent sequences of topologically nilpotent elements in ~$S.$

Alternatively, we claim that ~$A_{\infty}$ is canonically isomorphic to ~$W(A_{\infty}/pA_{\infty}),$ the Witt vectors of the residue ring ~$A_{\infty}/pA_{\infty},$ and hence ~$A_{\infty}$ is closely connected to constructions in ~$p$-adic Hodge theory. The fact that ~$A_{\infty} \cong W(A_{\infty}/pA_{\infty})$ follows from the observation that ~$A_{\infty}$ is a strict ~$p$-ring, i.e. that ~$A_{\infty}$ is complete and Hausdorff for the ~$p$-adic topology, ~$p$ is not a zero divisor in ~$A_{\infty},$ and ~$A_{\infty}/pA_{\infty}$ is perfect \cite{H}. Of these three criterion only the third is not immediately obvious for ~$A_{\infty}.$ To see that ~$A_{\infty}/pA_{\infty}$ is perfect, we recall:

\begin{theorem}\label{mod p lubin} \cite[Corollary 6.2.1, p. 343]{L} Let ~$k$ be a finite field, and let ~$\mathfrak{u},\mathfrak{f} \neq 0$ be invertible
and non-invertible, respectively, in ~$S_0(k)$, commuting with each other. Then either ~$\mathfrak{u}$ is a torsion element of ~$S_0(k)$ or ~$\mathfrak{f}$ has the form ~$\mathfrak{f}(x) = \mathfrak{a}(x^{p^h})$ with ~$h\in\Z_{+}$ and ~$a\in S_0(k)$ is invertible. \end{theorem}

\noindent In our case, theorem \ref{mod p lubin} implies ~$f (x) \equiv a(x^p)\mod p$ for some invertible series ~$a\in S_0(\F_p).$ It follows: 
\begin{align*}
A_{\infty}/pA_{\infty} &\cong A^{\circ}_{\infty}/pA^{\circ}_{\infty} \\ &\cong \displaystyle\varinjlim A_i/pA_i \\ &\cong \displaystyle\varinjlim_{x_i \mapsto a(x_{i+1}^p)} \F_p\llbracket      x_i\rrbracket \\ &\cong \displaystyle\varinjlim_{x_i \mapsto a(x_{i+1})^p} \F_p\llbracket      x_i\rrbracket \\ &\cong \displaystyle\varinjlim_{a^{\circ i}(x_i) \mapsto (a^{\circ i+1}(x_{i+1}))^p} \F_p\llbracket      a^{\circ i}(x_i)\rrbracket \\ &\cong \displaystyle\varinjlim_{y_i \mapsto y_{i+1}^p} \F_p\llbracket      y_i\rrbracket, \end{align*} where ~$y_i := a^{\circ i}(x_i).$ As this final ring is perfect the claim follows.  

Our goal is to find an injection ~$A_{\infty} \hookrightarrow \sA^+$ with particularly nice properties. Since ~$A_{\infty}$ is a strict ~$p$-ring, any injection ~$A_{\infty}/pA_{\infty} \hookrightarrow \sE^+$ lifts canonically to an injection ~$A_{\infty} \hookrightarrow \sA^+.$ Let ~$\Pi := (\pi_1,\pi_2,\pi_3, \dots)$ be the ~$f$-consistent sequence of elements in ~$\mathfrak{m}_{\C_p}$ fixed in section \ref{sect: a character}.  Then, as ~$f(x) \equiv a(x^p) \mod p,$ we observe for each ~$i \in \mathbb{Z}_+$ the series ~$\hat{\pi}_i := \langle a^{\circ k-i}(\pi_k) \mod p \rangle_{k \geq i} \in \sE^+.$ Define ~$\hat{\Pi} := (\hat{\pi}_1, \hat{\pi}_2, \hat{\pi}_3, \ldots).$  The sequence ~$\hat{\Pi}$ is ~$f$-consistent and consists of  topologically nilpotent elements of ~$\sE^+.$ Let ~$\phi_{\hat{\Pi}}: A_{\infty} \rightarrow \sE^+$ be the homomorphism associated to ~$\hat{\Pi}.$ 

\begin{prop} The kernel of the map ~$\phi_{\hat{\Pi}}: A_{\infty} \rightarrow \sE^+$ is ~$pA_{\infty}.$ The induced injection ~$\phi_{\hat{\Pi}}: A_{\infty}/pA_{\infty} \rightarrow \sE^+$ identifies ~$A_{\infty}/pA_{\infty}$ with a ~$G_K$-stable subring of ~$R$ such that for all ~$\sigma \in G_K$ and positive integers ~$i$ the equality  ~$\sigma(\phi_{\hat{\Pi}}(x_i)) = \phi_{\hat{\Pi}}([\chi_f(\sigma)]_f(x_i))$ holds.
\end{prop}
\begin{proof} The ring ~$\sE^+$ has characteristic ~$p$ and therefore ~$\phi_{\hat{\Pi}}$ factors through ~$$A_{\infty}/pA_{\infty} \cong A_{\infty}^{\circ}/pA_{\infty}^{\circ} \cong \varinjlim A_i/pA_i.$$ To prove the proposition, we show that the restriction of the map induced by ~$\phi_{\hat{\Pi}}$ to ~$A_i/pA_i$ is an injection and satisfies ~$\sigma(\phi_{\hat{\Pi}}(x_i)) = \phi_{\hat{\Pi}}([\chi_f(\sigma)]_f(x_i))$ for all ~$\sigma \in G_K.$ 

The ring ~$A_i/pA_i$ is isomorphic to ~$\F_p\llbracket      x_i\rrbracket.$ Therefore, any homomorphism out of ~$A_i/pA_i$ is either injective or has finite image. We claim the latter is false for the map induced by ~$\phi_{\hat{\Pi}}.$ To see this, consider ~$S_0(\F_p)$ acting on ~$\phi_{\hat{\Pi}}(A_i/pA_i).$ We claim that the stabilizer of ~$\phi_{\hat{\Pi}}(x_i)$ is trivial and hence ~$\phi_{\hat{\Pi}}(x_i)$ has infinite orbit. Let ~$z \in S_0(\F_p)$ be a nontrivial element and ~$n = deg(z(x) - x).$ Then ~$$z(\phi_{\hat{\Pi}}(x_i)) - x_i = \langle z(a^{\circ k-i}(\pi_k)) - \pi_k \mod p \rangle_{k \geq i}.$$ 
For any positive integer ~$k,$ the coordinate ~$z(a^{\circ k-i}(\pi_k)) - \pi_k \mod p$ is the image under the reduction map ~$\mathcal{O}_{\C_p} \rightarrow \mathcal{O}_{\C_p}/p$ of an element of ~$p$-adic valuation ~$\frac{n}{p^{k} - p^{k-1}}.$ Hence for ~$k \gg 0,$ we have ~$z(a^{\circ k-i}(\pi_k)) - \pi_k \mod p$ is nonzero. It follows ~$z$ does not stabilize ~$\phi_{\hat{\Pi}}(x_i).$ 

The proof that ~$G_K$ acts in the desired way on ~$\phi_{\hat{\Pi}}(x_i)$ is by direct calculation. Let ~$\sigma \in G_K.$ Then \begin{align*} \sigma(\phi_{\hat{\Pi}}(x_i)) &= \sigma(\hat{\pi}_i) \\ &= \langle \sigma(a^{\circ k-i}(\pi_k)) \mod p \rangle_{k \geq i} \\ &= \langle a^{\circ k-i}(\sigma(\pi_k)) \mod p \rangle_{k \geq i} \\ &= \langle a^{\circ k-i}([\chi_f(\sigma)]_f(\pi_i)) \mod p \rangle_{k \geq i} \\ &= \langle [\chi_f(\sigma)]_f(a^{\circ k-i}(\pi_i)) \mod p \rangle_{k \geq i} \\ & = [\chi_f(\sigma)]_f(\hat{\pi}) \\ &= \phi_{\hat{\Pi}}([\chi_f(\sigma)]_f(x_i)). \end{align*}\end{proof}

Because ~$A_{\infty}$ is a strict ~$p$-ring, the injection induced by ~$\phi_{\hat{\Pi}}:A_{\infty}/pA_{\infty} \hookrightarrow \sE^+$ lifts canonically to an injection ~$W(\phi_{\hat{\Pi}}):A_{\infty} \hookrightarrow \sA^+.$ Henceforth, we will identify ~$A_{\infty}$ with its image under ~$W(\phi_{\hat{\Pi}}).$ The ~$G_K$ action on ~$A_{\infty}/pA_{\infty}$ lifts functorially to a ~$G_K$ action on ~$A_{\infty};$ furthermore, the map ~$W(\phi_{\hat{\Pi}})$ is ~$G_K$-equivariant and identifies ~$A_{\infty}$ with a ~$G_K$-stable subring of ~$\sA^+.$ Our next goal is to precisely describe the action of ~$G_K$ on ~$A_{\infty}.$ The following theorem provides the rigidity needed to understand these lifts. 

Let ~$\F_q$ be a finite extension of ~$\F_p$ and consider an invertible series ~$\omega \in S_0(\F_q)$ such that ~$\omega'(0) = 1.$  The absolute ramification index of ~$\omega$ is defined to be the limit 

$$e(\omega) :=  \lim_{n \rightarrow \infty} (p-1)v_x(\omega^{\circ p^n}(x) -x)/p^{n+1}.$$ 

\begin{theorem}\label{LMSS}\cite[Proposition~5.5]{LMS}. Assume ~$e(\omega) < \infty,$ then the separable normalizer of ~$\omega^{\circ \Z_p}$ in ~$S_0(\F_q)$ is a finite extension ~$\omega^{\circ \Z_p}$ by a group of order dividing ~$e(\omega).$ 
\end{theorem} 

\begin{lemma}\label{power of a} There exists a positive integer ~$m$ such that ~$a^{\circ N} \in u^{\circ \Z_p} \mod p.$ 
\end{lemma}

\begin{proof}  Note as ~$f$ and ~$u$ commute, so too do ~$a$ and the reduction of ~$u$ modulo ~$p$.  One observes, upon considering the Newton polygon ~$u^{\circ n}(x) - x,$ that ~$v_x(u^{\circ n}(x) - x) = |1-u'(0)^n|_p.$ By assumption \ref{u assumption}, the absolute ramification index of ~$u$ is therefore finite. The result follows by Theorem \ref{LMSS}.
\end{proof}

\begin{prop} Let ~$\sigma \in G_K.$ Then ~$\sigma(x_i) = [\chi_f(\sigma)]_f(x_i).$
\end{prop}
\begin{proof} Consider the sequence ~$$[\chi_f(\sigma)]_f(\Pi^{univ}) := ([\chi_f(\sigma)]_f(x_i))_{i\in\Z^{+}}$$ of elements in ~$A_{\infty}.$ Because ~$f$ and ~$[\chi_f(\sigma)]_f$ commute, the sequence ~$[\chi_f(\sigma)]_f(\Pi^{univ})$ is ~$f$-consistent. Hence, by the universal property of ~$A_{\infty},$ there exists an endomorphism ~$[\chi_f(\sigma)]_f^*:= \phi_{[\chi_f(\sigma)]_f(\Pi^{univ})}$ of ~$A_{\infty}$ which maps ~$x_i$ to ~$[\chi_f(\sigma)]_f(x_i).$ 

We wish to show that the Witt lift ~$W(\sigma)$ is equal to ~$[\chi_f(\sigma)]_f^*.$ The former map is defined by its action mod ~$p$ on Teichm{\"u}ller representatives. Denote the Teichm{\"u}ller mapping by ~$\{*\}: A_\infty/pA_{\infty} \rightarrow A_\infty.$ For ~$\overline{\alpha} \in  A_\infty/pA_{\infty}$ the element ~$\{\overline{\alpha}\}$ can be defined as follows: let ~$\alpha_n$ be any lift of ~$\overline{\alpha}^{1/p^n},$ then ~$\{\overline{\alpha}\} := \lim_{n \rightarrow \infty} \alpha_n^{p^n}.$ The map ~$\{*\}$ is well defined as this limit converges and is independent of the choice of lifts ~$\mathbb{\alpha}_n$. The automorphism ~$W(\sigma)$ is the unique map such that ~$W(\sigma)\{\overline{\alpha}\} = \{\sigma(\overline{\alpha})\}.$

We claim the set of Teichm{\"u}ller lifts of ~$x_i \mod p$ together topologically generate ~$A_{\infty}$ (in the ~$(p,x_1)$-adic topology) as a ~$\Z_p$-algebra. To see this, let ~$A_{\infty}^{tm}$ be the sub-algebra topologically generated by these elements. We show ~$A_{\infty}^{tm} =A_{\infty}.$ First note that the residue rings ~$A_{\infty}^{tm}/pA_{\infty}^{tm}$ and ~$A_{\infty}/pA_{\infty}$ are equal. Next observe that ~$A_{\infty}^{tm}$ is a strict ~$p$-ring. To see this, note ~$A_{\infty}^{tm}$ is closed in the ~$p$-adic topology on ~$A_{\infty}$ and is therefore ~$p$-adically complete and Hausdorff, ~$A_{\infty}^{tm}$  is a sub-algebra of ~$A_{\infty}$ and hence ~$p$ is not a zero divisor in ~$A_{\infty}^{tm},$ and ~$A_{\infty}^{tm}/pA_{\infty}^{tm} = A_{\infty}/pA_{\infty}$ and hence the residue ring ~$A_{\infty}^{tm}/pA_{\infty}^{tm}$ is perfect. It follows the image of ~$A_{\infty}^{tm}/pA_{\infty}^{tm} = A_{\infty}/pA_{\infty}$ under the Teichm{\"u}ller map lies in  ~$A_{\infty}^{tm}.$ But ~$A_{\infty}$ (in the ~$p$-adic topology) is topologically generated as a ~$\Z_p$-algebra by the image of ~$A_{\infty}/pA_{\infty}$ under the Teichm{\"u}ller map. We conclude ~$A_{\infty}^{tm} =A_{\infty}.$ 

Therefore, to show  ~$W(\sigma) = [\chi_f(\sigma)]_f^*,$ it is enough to show ~$W(\sigma(\{x_i\})) := \{\sigma(x_i)\}$ is equal to ~$[\chi_f(\sigma)]_f^*(\{x_i\})$ for all ~$i\in\Z^{+}.$ Fix a positive integer ~$i.$ By lemma \ref{power of a}, there exists an integer ~$N>0$ and a ~$p$-adic number ~$k$ such that ~$a^{\circ N} = u^{\circ k}.$ Hence, as ~$f(x) \equiv a(x^p) \mod p,$ it follows for every positive integer ~$m$ the element ~$u^{\circ -km}(x_{Nm+i})$ is a lift of ~$x_i^{p^{1/Nm}} \mod p.$ Similarly, ~$u^{\circ -km}([\chi_f(\sigma)]_f(x_{Nm+i}))$ is a lift of ~$(\sigma(x_i))^{p^{1/Nm}} \mod p.$ Using these lifts we will compare the action of ~$W(\sigma)$  and ~$[\chi_f(\sigma)]_f^*$ on ~$\{x_i\}.$ Observe 

\begin{align*}[\chi_f(\sigma)]_f^*(\{x_i\}) &= \lim_{m \rightarrow \infty} ([\chi_f(\sigma)]_f^*(u^{\circ -km}(x_{Nm+i})))^{p^{Nm}}
\\&=  \lim_{m \rightarrow \infty} (u^{\circ -km}([\chi_f(\sigma)]_f^*(x_{Nm+i})))^{p^{Nm}}
\\&=  \lim_{m \rightarrow \infty} (u^{\circ -km}([\chi_f(\sigma)]_f(x_{Nm+i})))^{p^{Nm}}
\\&= W(\sigma)(\{x_i\}). 
\end{align*} \end{proof}

\begin{remark} \emph{ The ring ~$A_{\infty}$ is not an unfamiliar object in the study of formal groups (or more generally ~$p$-divisible groups). If ~$f$ is the endomorphism of a formal group ~$F := \mathrm{Spf}(\Z_p\llbracket      x_1\rrbracket),$ the completion of  ~$A_{\infty}$ is the ring of global functions on the universal cover of F (see \cite[Section 3.1]{SW}).}  
\end{remark}

\subsection{The ~$p$-adic Regularity of ~$\chi_f$}\label{sect:regularity}

In this section, we will examine the image the sequence ~$\langle x_i \rangle$ under ~$\log_f.$ To ensure convergence, we will appeal to the following lemma of Berger: 

\begin{lemma}\label{Ber} \cite[Lemma 3.2]{B} Let ~$E$ be a finite extension of ~$\Q_p$ and take ~$L(X) \in E\llbracket      X\rrbracket.$ If ~$x \in \mathbf{B}_{dR}^{+},$ then the series ~$L(x)$ converges in ~$\mathbf{B}_{dR}^{+}$ if and only if ~$L(\theta(x))$ converges in ~$\C_p.$
\end{lemma}

By construction, the image of ~$A_{\infty}$ lies in ~$\sA^+$ and therefore ~$\theta(x_i) \in \mathcal{O}_{\C_p}.$ The map from ~$A_{\infty}$ to ~$\sA^+$ is a lift of the map ~$\phi_{\hat{\Pi}}: A_{\infty} \rightarrow \sE^+$ which sends ~$x_i$ to ~$\hat{\pi}_i :=  \langle a^{\circ k-i}(\pi_k) \mod p \rangle_{k \geq i}.$  Therefore, ~$\theta(x_i) \equiv \pi_i \mod p$ and hence ~$\theta(x_i) \in \mathfrak{m}_{\C_p}.$ Now the series ~$\log_f(X)$ converges on ~$\mathfrak{m}_{\C_p}$ and so we see from lemma \ref{Ber} that ~$\log_f(x_i)$ converges in ~$\mathbf{B}_{dR}^{+}$ for all ~$i\in\Z^+.$ 

Set ~$x_0 := f(x_1).$ Then ~$\theta(x_0) \in \mathfrak{m}_{\C_p}$ and hence ~$\log_f(x_0)$ converges in ~$\mathbf{B}_{dR}^{+}.$ We define ~$t_f := \log_f(x_0).$ We note that 
$$(f'(0))^{n+1} \log_f(x_n) = \log_f(f^{\circ n+1}(x_n)) = \log_f(f(x_1)) = t_f$$ and therefore the values ~$\log_f(x_n)$ are ~$\Q_p$-multiples of ~$t_f.$ We call ~$t_f$ a fundamental period of ~$f.$ The reader should be warned that ~$t_f$ depends not only on ~$f$ but also our choice of ~$f$-consistent sequence ~$\Pi$. If ~$f$ is an endomorphism of a formal group defined over ~$\Z_p$ any two fundamental periods differ by a ~$p$-adic unit.  

How does ~$G_K$ act on ~$t_f?$ Let ~$\sigma \in G_K,$ then 
\begin{align*} \sigma(t_f) &= \sigma \log_f(f(x_1)) 
\\&= \log_f(f(\sigma(x_1))) 
\\&= \log_f(f\circ[\chi_f(\sigma)]_f(x_1)) 
\\&= f'(0)\chi_f(\sigma)\log_f(x_1) 
\\&= \chi_f(\sigma)t_f
\end{align*} Therefore, assuming ~$t_f \neq 0,$ the period ~$t_f$ generates a ~$G_K$-stable ~$\Q_p$-line of ~$\mathbf{B}_{dR}^{+}$ on which ~$G_K$ acts through ~$\chi_f.$ From this it follows ~$\chi_f$ is de Rham of some positive weight. Our first proposition of this section shows that this is in fact the case.    

\begin{prop} ~$\chi_f$ is de Rham of weight 1. 
\end{prop}

\begin{proof} The claim will follow if we can show ~$t_f \in \mathrm{Fil}^1 \mathrm{\mathbf{B}}_{dR}^{+}\setminus \mathrm{Fil}^2 \mathrm{\mathbf{B}}_{dR}^{+}.$ 

We begin by showing ~$t_f \in \mathrm{{Fil}}^1 \mathrm{\mathbf{B}}_{dR}^{+}.$ Assume this were not the case. Then ~$\chi_f$ would be Rham of weight 0. However, all such representations are potentially unramified and ~$K_{\infty},$ the fixed field of ~$\ker \chi_f,$ is an infinitely ramified ~$\Z_p$-extension of ~$K.$ It follows that ~$t_f \in \mathrm{Fil}^1 \mathrm{\mathbf{B}}_{dR}^{+}.$ 

Next we show ~$t_f \not\in \mathrm{{Fil}}^2 \mathrm{\mathbf{B}}_{dR}^{+}.$ Since ~$0 = \theta(t_f) = \log_f(\theta(x_0)),$ we observe ~$\theta(x_0)$ is a root of ~$\log_f.$ The quotient ~$\mathrm{\mathbf{B}}_{dR}^{+}/\mathrm{Fil}^2\mathrm{\mathbf{B}}_{dR}^{+}$ is a square zero extension of ~$\mathrm{\mathbf{B}}_{dR}^{+}/\mathrm{Fil}^1\mathrm{\mathbf{B}}_{dR}^{+} \cong \C_p.$ Since all roots of ~$\log_f$ in ~$\mathfrak{m}_{\C_p}$ are simple, we conclude ~$t_f \not\in \mathrm{{Fil}}^2 \mathrm{\mathbf{B}}_{dR}^{+}.$
\end{proof}

\begin{lemma}\label{dlog} The derivative of the logarithm, ~$\log_f',$ is an element of ~$\Z_p\llbracket      x\rrbracket.$
\end{lemma}

\begin{proof} Since the series ~$\log_f$ converges on ~$\mathfrak{m}_{\C_p},$ so too does ~$\log_f'.$ The Newton polygon of ~$\log_f'$ contains the point ~$(0,0).$ Hence, ~$\log_f'\in \Z_p\llbracket      x\rrbracket$ if and only if ~$\log_f'$ has no roots in ~$\mathfrak{m}_{\C_p}.$ Assume for the sake of contradiction ~$\log_f'$ had a root ~$\pi$ in ~$\mathfrak{m}_{\C_p}.$ Then \begin{equation}\label{eqn : derivative} \log_f'(u(\pi))u'(\pi) = 
(\log_f \circ u)'(\pi) = u'(0)\log_f'(\pi) = 0.
\end{equation} The power series ~$u'(x)$ is invertible in ~$\Z_p\llbracket      x\rrbracket,$ so equation \ref{eqn : derivative} implies  ~$\log_f'(u(\pi)) = 0.$ Hence, ~$u^{\circ \Z_p}$ acts on the set of roots of ~$\log_f'$. Since ~$u$ preserves valuation and ~$\log_f'$ has only finitely many roots of any fixed valuation, some power of ~$u$ fixes ~$\pi.$ From the equality ~$\Lambda_u = \Lambda_f,$ we conclude the element ~$\pi$ is a root of ~$\log_f$. But this is a contradiction as  ~$\log_f$ has simple roots.  
\end{proof}

\begin{prop} The period ~$t_f$ is an element of ~$\mathbf{A}_{cris},$ and ~$\chi_f$ is crystalline. 
\end{prop}
\begin{proof} By lemma \ref{dlog}:  ~$$\log_f(x) = \sum_{i=1}^{\infty} \frac{a_n}{n} x^n$$ where ~$a_n \in \Z_p$ for all ~$n\in \Z^+.$ Therefore,  ~$$t_f = \sum_{i=1}^{\infty} a_n(n-1)!\frac{x_0^n}{n!}.$$ It follows if ~$\theta(x_0) = 0,$ then ~$t_f \in \mathbf{A}_{cris}.$  To see ~$\theta(x_0) = 0,$ note that ~$\theta(x_0)$ is a root of ~$\log_f$ in ~$\mathfrak{m}_{\C_p}$ and ~$$\theta(x_0) \equiv \theta(f(x_1))  \equiv f(\pi_1) \equiv 0 \mod p.$$ As ~$0$ is the unique root of ~$\log_f$ of ~$p$-adic valuation greater than ~$1,$ the element ~$\theta(x_0) = 0$ and the claim follows. \end{proof}

\section{The Main Theorem}\label{the main theorem}

\subsection{Constructing the Formal Group}\label{sect:formal group} Let ~$V_f$ be the ~$\Q_p$-line of ~$\mathbf{B}_{cris}^+$ generated by the period ~$t_f.$ In section \ref{sect:regularity}, we deduced that ~$V_f$ is a ~$G_K$-sub-representation of ~$\mathbf{B}_{cris}^+$ isomorphic to ~$\chi_f$. Since ~$K/\Q_p$ is totally ramified, any ~$G_K$-representation on a vector space ~$V/\Q_p$ has multiplicity in ~$\mathbf{B}_{cris}^+$ at most one. Therefore, ~$V_f$ is the unique ~$\Q_p$-line of ~$\mathbf{B}_{cris}^+$ isomorphic to ~$\chi_f.$ It follows ~$V_f$ is preserved by crystalline Frobenius. Let ~$\pi_f \in \Q_p$ be the eigenvalue of crystalline Frobenius acting on ~$V_f.$ Because ~$V_f$ has weight one, the ~$p$-adic valuation of ~$\pi_f$ is equal to ~$1.$ 

The next lemma is fundamental. We will show that the equality ~$\varphi(\log_f(x_0)) = \pi_f\log_f(x_0)$ implies that ~$\varphi(x_0) = [\pi_f]_f(x_0).$ This will be enough to show ~$[\pi_f]_f \in S_0(\Z_p)$ and satisfies ~$[\pi_f]_f(x) \equiv x^p \mod p,$ from which we will deduce the that ~$f$ and ~$u$ are endomorphisms of an integral formal group by a lemma of Lubin and Tate. 

\begin{lemma} The power series ~$[\pi_f]_f \in \pi_fx + x^2\Z_p\llbracket      x\rrbracket$ and satisfies ~$[\pi_f]_f(x) \equiv x^p \mod p.$
\end{lemma}

\begin{proof} First note that by construction ~$x_0 \in \sA^+ \subseteq \mathbf{B}_{cris}^+$ and so ~$x_0$ is acted upon by crystalline Frobenius. Consider ~$\varphi(x_0).$ We claim ~$\theta(\varphi(x_0))= 0.$ To see this, observe  ~$$\log_f(\theta(\varphi(x_0))) = \theta(\varphi(\log_f(x_0)) = \theta(\pi_ft_f) = 0.$$ Hence, ~$\theta(\varphi(x_0))$ is a root of ~$\log_f(x).$ As ~$x_0 \in  \sA^+$ and ~$\theta(x_0)=0,$ the image of crystalline Frobenius satisfies ~$\theta(\varphi(x_0)) \in p\mathcal{O}_{\C_p}.$ The unique root of ~$\log_f(x)$ in ~$p\mathcal{O}_{\C_p}$ is ~$0,$ therefore we conclude ~$\theta(\varphi(x_0))= 0.$

Because ~$\theta(x_0) = 0$ and ~$\theta(\varphi(x_0))= 0,$ any series ~$e(x)\in \Q_p\llbracket      x\rrbracket$ converges in ~$\mathrm{\mathbf{B}}_{dR}^{+}$ when evaluated at ~$x_0$ or ~$\varphi(x_0).$  From this we deduce the equality ~$$\log_f(\varphi(x_0)) = \varphi(t_f) = \pi_ft_f = \log_f([\pi_f]_f(x_0)).$$ Additionally, we may apply ~$\log_f^{-1}$ to both sides, obtain convergent results and deduce ~$\varphi(x_0)= [\pi_f]_f(x_0).$ 

Write ~$$[\pi_f]_f(x) \in \pi_fx + \sum_{i=2}^{\infty} a_nx^n$$ where ~$a_n \in \Q_p.$ We prove by induction on ~$n$ that ~$a_n \in \Z_p.$ Let ~$n\geq 2$ and assume ~$a_i \in \Z_p$ for all ~$i<n.$ Set ~$$[\pi_f]_f^{<n} := \pi_fx + \sum_{i=2}^{n-1} a_nx^n.$$ Observe the equality
$$[\pi_f]_f^{<n}(x_0) \equiv \varphi(x_0) \equiv x_0^p \mod Fil^{n}\mathbf{B}_{dR}^+ \cap \sA^+ + p\sA^+.$$ Hence, as the homomorphism ~$\Z_p\llbracket      x\rrbracket \rightarrow \sA^+$ given by mapping ~$x \mapsto x_0$ is injective and ~$(x)$ is the pullback of ~$Fil^{n}\mathbf{B}_{dR}^+ \cap \sA^+,$ it must be the case that ~$[\pi_f]_f^{<n}(x) \equiv x^p \mod p.$ Express ~$f(x) = f'(0)x + f^{\ge 2}(x)$ where ~$f^{\ge 2}(x) \in x^2\Z_p\llbracket      x\rrbracket.$ Note that$$ f\circ[\pi_f]_f(x) \equiv f([\pi_f]_f^{<n}(x)) + a_nf'(0)x^n \mod x^{n+1}$$ and ~$$ [\pi_f]_f(f(x)) \equiv [\pi_f]_f^{<n}(f(x)) + a_nf'(0)^nx^n \mod x^{n+1}.$$ As ~$f(x)$ and ~$[\pi_f](x)$ commute, 
\begin{equation} \label{eqn : stuff} f([\pi_f]_f^{<n}(x)) - [\pi_f]_f^{<n}(f(x)) \equiv a_n(f'(0) -f'(0)^n)x^n \mod x^{n+1} \end{equation}
The left hand side of \eqref{eqn : stuff} is a power series over ~$\Z_p$ and reduces modulo ~$p$ to ~$$f([\pi_f]_f^{<n}(x)) - [\pi_f]_f^{<n} \equiv f(x^p) -f(x)^p \equiv 0 \mod p.$$ Therefore, the right hand side of \eqref{eqn : stuff} is a power series over ~$p\Z_p.$ Since ~$v_p(f'(0) -f'(0)^n) = 1,$ it follows ~$a_n \in \Z_p.$ By induction we conclude ~$[\pi_f]_f$ is an element of ~$\pi_fx + x^2\Z_p\llbracket      x\rrbracket.$ 
From the equality ~$\varphi(x_0)= [\pi_f]_f(x_0),$ we observe ~$[\pi_f]_f(x) \equiv x^p \mod p.$ \end{proof}

\samepage{ \begin{theorem} Let ~$f$ and ~$u$ be a commuting pair of elements in ~$S_0(\Z_p)$. If
\begin{itemize}
\item ~$f'(0)$ is prime in ~$\Z_p$ and ~$f$ has exactly ~$p$ roots in ~$\mathfrak{m}_{\C_p},$
\item ~$u$ is invertible and has infinite order,
\end{itemize} 
then there exists a unique formal group law ~$F/\Z_p$ such that ~$f,u \in \mathrm{End}_{\Z_p}(F).$ The formal group law ~$F$ is isomorphic to ~$\widehat{\G}_m$ over the ring of integers of the maximal unramified extension of ~$\Q_p.$       
\end{theorem}}

\begin{proof} By \cite[Theorem~1.1]{LT}, the series ~$[\pi_f]_f$ is an endomorphism of a unique height one formal group over ~$F/\Z_p.$ Furthermore, ~$f$ and ~$u$ are endomorphisms of ~$F$.  The second half of the claim follows as all height one formal groups over ~$\Z_p$ are forms of ~$\G_m$ and are trivialized over the ring of integers of the maximal unramified extension of ~$\Q_p.$       
\end{proof} 

\section{Acknowledgements}
I would like to thank Frank Calegari, Vlad Serban, and Paul VanKoughnett for their helpful edits of an earlier draft of this paper.

\end{document}